\documentclass[11pt]{amsart}

\usepackage{amssymb}
\usepackage{amsmath,amssymb,epsf,wrapfig,multicol,epic,eepic,trig,mathrsfs,dsfont,color,graphicx}
 \usepackage[abs]{overpic}
 \usepackage{here}

\theoremstyle{plain}
\newtheorem{thm}{Theorem}
\newtheorem{cor}[thm]{Corollary}
\newtheorem{lem}[thm]{Lemma}
\newtheorem{prop}[thm]{Proposition}

\theoremstyle{remark}
\newtheorem{rem}[thm]{Remark}

\newcommand{\cal}{\mathcal}

\def\hsymbu#1{\smash{\lower1.7ex\hbox{\huge$#1$}}}

\def\rmoveio#1#2{
\setlength{\unitlength}{#1}
\begin{picture}(50,30)
\put(5,0){\line(0,1){30}}

{\allinethickness{.8pt}
\put(10,15){\vector(1,0){13}}
\put(23,15){\vector(-1,0){13}}}

\qbezier(25,0)(25,20)(40,20)
\qbezier(40,20)(45,20)(45,15)
\qbezier(45,15)(45,10)(40,10)
\qbezier(40,10)(35,10)(31,14)
\qbezier(28,17)(25,25)(25,30)

\ifnum#2=2
\put(3,28){\path(0,0)(2,2)(4,0)}
\put(28,28){\path(0,0)(2,2)(4,0)}
\put(38,15){\makebox{${\Huge c_{1}}$}}
\fi

\end{picture}
}
\def\rmoveiio#1#2{
\setlength{\unitlength}{#1}
\begin{picture}(60,30)
\put(5,0){\line(0,1){30}}
\put(15,0){\line(0,1){30}}

{\allinethickness{.8pt}
\put(20,15){\vector(1,0){15}}
\put(35,15){\vector(-1,0){15}}}

\qbezier(40,0)(42,1)(47,3)
\qbezier(52,6)(68,15)(52,24)
\qbezier(47,27)(42,30)(40,30)

\qbezier(60,0)(20,15)(60,30)

\ifnum#2=2
\put(2,27){\path(0,0)(3,3)(6,0)}
\put(12,27){\path(0,0)(3,3)(6,0)}
\put(40,27){\path(0,0)(0,3)(3,3)}
\put(60,27){\path(0,0)(0,3)(-3,3)}
\put(47,16){\makebox{${\Huge c_{1}}$}}
\put(47,10){\makebox{${\Huge c_{2}}$}}
\fi

\ifnum#2=3
\put(2,2){\path(0,0)(3,-3)(6,0)}
\put(12,27){\path(0,0)(3,3)(6,0)}
\put(40,3){\path(0,0)(0,-3)(3,-3)}
\put(60,27){\path(0,0)(0,3)(-3,3)}
\put(47,16){\makebox{${\Huge c_{1}}$}}
\put(47,10){\makebox{${\Huge c_{2}}$}}
\fi

\end{picture}
}
\def\rmoveiiio#1#2{
\setlength{\unitlength}{#1}
\begin{picture}(75,30)
\put(0,0){\line(1,1){15}}
\qbezier(15,15)(20,20)(20,30)

\put(10,0){\line(-1,1){4}}
\qbezier(4,6)(-5,15)(5,25)
\put(5,25){\line(1,1){5}}

\qbezier(20,0)(20,10)(16,14)
\put(14,16){\line(-1,1){8}}
\put(4,26){\line(-1,1){4}}

{\allinethickness{.8pt}
\put(25,15){\vector(1,0){15}}
\put(40,15){\vector(-1,0){15}}}

\qbezier(50,0)(50,10)(55,15)
\put(55,15){\line(1,1){15}}

\put(60,0){\line(1,1){5}}
\qbezier(65,5)(75,15)(66,24)
\put(64,26){\line(-1,1){4}}

\put(70,0){\line(-1,1){4}}
\put(64,6){\line(-1,1){8}}
\qbezier(54,16)(50,20)(50,30)

\ifnum#2=2
\put(0,27){\path(0,0)(0,3)(3,3)}
\put(7,30){\path(0,0)(3,0)(3,-3)}
\put(17,27){\path(0,0)(3,3)(6,0)}

\put(10,23){\makebox{${\Huge c_{1}}$}}
\put(10,3){\makebox{${\Huge c_{2}}$}}
\put(20,13){\makebox{${\Huge c_{3}}$}}

\put(47,27){\path(0,0)(3,3)(6,0)}
\put(60,27){\path(0,0)(0,3)(3,3)}
\put(67,30){\path(0,0)(3,0)(3,-3)}

\put(70,23){\makebox{${\Huge c'_{2}}$}}
\put(70,3){\makebox{${\Huge c'_{1}}$}}
\put(60,13){\makebox{${\Huge c'_{3}}$}}
\fi
\end{picture}
}
\def\rmovevio#1#2{
\setlength{\unitlength}{#1}
\begin{picture}(50,30)
\put(5,0){\line(0,1){30}}

{\allinethickness{.8pt}
\put(10,15){\vector(1,0){15}}
\put(25,15){\vector(-1,0){15}}}

\qbezier(30,0)(30,20)(45,20)
\qbezier(45,20)(50,20)(50,15)
\qbezier(50,15)(50,10)(45,10)
\qbezier(45,10)(30,10)(30,30)

\put(34,15){\circle{5}}

\ifnum#2=2
\put(3,28){\path(0,0)(2,2)(4,0)}
\put(28,28){\path(0,0)(2,2)(4,0)}
\put(38,15){\makebox{${\Huge c_{1}}$}}
\fi

\end{picture}
}
\def\rmoveviio#1#2{
\setlength{\unitlength}{#1}
\begin{picture}(60,30)
\put(5,0){\line(0,1){30}}
\put(15,0){\line(0,1){30}}

{\allinethickness{.8pt}
\put(20,15){\vector(1,0){15}}
\put(35,15){\vector(-1,0){15}}}

\qbezier(40,0)(80,15)(40,30)
\qbezier(60,0)(20,15)(60,30)
\put(50,4){\circle{5}}
\put(50,25){\circle{5}}

\ifnum#2=2
\put(2,27){\path(0,0)(3,3)(6,0)}
\put(12,27){\path(0,0)(3,3)(6,0)}
\put(40,27){\path(0,0)(0,3)(3,3)}
\put(60,27){\path(0,0)(0,3)(-3,3)}
\put(47,15){\makebox{${\Huge c_{1}}$}}
\put(47,10){\makebox{${\Huge c_{2}}$}}
\fi

\ifnum#2=3
\put(2,2){\path(0,0)(3,-3)(6,0)}
\put(12,27){\path(0,0)(3,3)(6,0)}
\put(40,3){\path(0,0)(0,-3)(3,-3)}
\put(60,27){\path(0,0)(0,3)(-3,3)}
\put(47,15){\makebox{${\Huge c_{1}}$}}
\put(47,10){\makebox{${\Huge c_{2}}$}}
\fi

\end{picture}
}
\def\rmoveviiio#1#2{
\setlength{\unitlength}{#1}
\begin{picture}(70,30)
\put(0,0){\line(1,1){15}}
\qbezier(15,15)(20,20)(20,30)

\put(10,0){\line(-1,1){5}}
\qbezier(5,5)(-5,15)(5,25)
\put(5,25){\line(1,1){5}}

\qbezier(20,0)(20,10)(15,15)
\put(15,15){\line(-1,1){15}}

\put(5,5){\circle{5}}
\put(15,15){\circle{5}}
\put(5,25){\circle{5}}

{\allinethickness{.8pt}
\put(28,15){\vector(1,0){14}}
\put(42,15){\vector(-1,0){14}}}

\qbezier(50,0)(50,10)(55,15)
\put(55,15){\line(1,1){15}}

\put(60,0){\line(1,1){5}}
\qbezier(65,5)(75,15)(65,25)
\put(65,25){\line(-1,1){5}}

\put(70,0){\line(-1,1){15}}
\qbezier(55,15)(50,20)(50,30)

\put(65,5){\circle{5}}
\put(55,15){\circle{5}}
\put(65,25){\circle{5}}

\ifnum#2=2
\put(0,27){\path(0,0)(0,3)(3,3)}
\put(7,30){\path(0,0)(3,0)(3,-3)}
\put(17,27){\path(0,0)(3,3)(6,0)}

\put(20,13){\makebox{${\Huge c_{1}}$}}

\put(47,27){\path(0,0)(3,3)(6,0)}
\put(60,27){\path(0,0)(0,3)(3,3)}
\put(67,30){\path(0,0)(3,0)(3,-3)}

\put(60,13){\makebox{${\Huge c'_{1}}$}}
\fi

\end{picture}
}
\def\rmovevivo#1#2{
\setlength{\unitlength}{#1}
\begin{picture}(70,30)
\put(0,0){\line(1,1){15}}
\qbezier(15,15)(20,20)(20,30)

\put(10,0){\line(-1,1){5}}
\qbezier(5,5)(-5,15)(5,25)
\put(5,25){\line(1,1){5}}

\qbezier(20,0)(20,10)(16,14)
\put(14,16){\line(-1,1){14}}

\put(5,5){\circle{5}}
\put(5,25){\circle{5}}

{\allinethickness{.8pt}
\put(28,15){\vector(1,0){14}}
\put(42,15){\vector(-1,0){14}}}

\qbezier(50,0)(50,10)(55,15)
\put(55,15){\line(1,1){15}}

\put(60,0){\line(1,1){5}}
\qbezier(65,5)(75,15)(65,25)
\put(65,25){\line(-1,1){5}}

\put(70,0){\line(-1,1){14}}
\qbezier(54,16)(50,20)(50,30)

\put(65,5){\circle{5}}
\put(65,25){\circle{5}}

\ifnum#2=2
\put(0,27){\path(0,0)(0,3)(3,3)}
\put(7,30){\path(0,0)(3,0)(3,-3)}
\put(17,27){\path(0,0)(3,3)(6,0)}

\put(20,13){\makebox{${\Huge c_{1}}$}}

\put(47,27){\path(0,0)(3,3)(6,0)}
\put(60,27){\path(0,0)(0,3)(3,3)}
\put(67,30){\path(0,0)(3,0)(3,-3)}

\put(60,13){\makebox{${\Huge c'_{1}}$}}
\fi

\end{picture}
}
\def\rmovetti#1{
\setlength{\unitlength}{#1}
\begin{picture}(60,20)
\put(0,0){\line(1,1){20}}
\put(20,0){\line(-1,1){20}}
\put(10,10){\circle{5}}
\put(17,13){\line(-1,1){4}}

{\allinethickness{.8pt}
\put(25,10){\vector(1,0){10}}
\put(35,10){\vector(-1,0){10}}}

\put(40,0){\line(1,1){20}}
\put(60,0){\line(-1,1){20}}
\put(50,10){\circle{5}}
\put(47,3){\line(-1,1){4}}
\end{picture}
}

\def\rmovetiio#1{
\setlength{\unitlength}{#1}
\begin{picture}(40,20)
\put(5,0){\line(0,1){20}}
{
\put(2.,7){\line(1,0){6}} 
\put(2.,13){\line(1,0){6}} }
{\allinethickness{.8pt}
\put(15,10){\vector(1,0){10}}
\put(25,10){\vector(-1,0){10}}}
\put(35,0){\line(0,1){20}}
\end{picture}
}
\def\rmovetiiio#1#2{
\setlength{\unitlength}{#1}
\begin{picture}(70,20)

\put(0,0){\line(1,1){20}}
\put(20,0){\line(-1,1){9}}
\put(9,11){\line(-1,1){9}}

\put(7,3){\line(-1,1){4}}
\put(13,3){\line(1,1){4}}
\put(3,13){\line(1,1){4}}
\put(17,13){\line(-1,1){4}}

{\allinethickness{.8pt}
\put(25,10){\vector(1,0){10}}
\put(35,10){\vector(-1,0){10}}}

\qbezier(40,5)(48,20)(54,11)
\qbezier(56, 9)(62,0)(70,15)

\qbezier(40,15)(48,0)(55,10)
\qbezier(55,10)(62,20)(70,5)

\put(43,10){\circle{5}}
\put(67,10){\circle{5}}

\ifnum#2=2
\put(0,17){\path(0,0)(0,3)(3,3)}
\put(17,20){\path(0,0)(3,0)(3,-3)}

\put(15,8){\makebox{${\Huge c_{1}}$}}

\put(40,12){\path(0,0)(0,3)(3,3)}
\put(67,15){\path(0,0)(3,0)(3,-3)}

\put(53,3){\makebox{${\Huge c'_{1}}$}}
\fi

\end{picture}
}
\def\abscrs#1{
\setlength{\unitlength}{#1}
\begin{picture}(75,30)
\allinethickness{1pt}
\put(0,5){\line(1,1){20}}
\put(0,25){\line(1,-1){8}}
\put(20,5){\line(-1,1){8}}
{\allinethickness{1pt}
\put(25,15){\vector(1,0){10}}
}
{\allinethickness{1pt}
\qbezier(50,3)(55,13)(60,3)
\qbezier(50,27)(55,17)(60,27)
\qbezier(42,10)(52,15)(42,20)
\qbezier(67,10)(57,15)(67,20)
}

{\allinethickness{2pt}
\put(45,5){\line(1,1){20}}
\put(45,25){\line(1,-1){8}}
\put(65,5){\line(-1,1){8}}
}

\end{picture}
}
\def\absvcrs#1{
\setlength{\unitlength}{#1}
\begin{picture}(75,30)
\allinethickness{1pt}
\put(0,5){\line(1,1){20}}
\put(0,25){\line(1,-1){20}}
\put(10,15){\circle{5}}
{\allinethickness{1pt}
\put(25,15){\vector(1,0){10}}
}

\put(40,5){\line(1,1){20}}
\put(40,25){\line(1,-1){10}}
\put(60,5){\line(-1,1){5}}

{\allinethickness{2pt}
\put(45,5){\line(1,1){20}}
\put(45,25){\line(1,-1){7}}
\put(65,5){\line(-1,1){7}}
}

\put(50,5){\line(1,1){20}}
\put(50,25){\line(1,-1){5}}
\put(70,5){\line(-1,1){10}}

\end{picture}
}
\def\abstwt#1{
\setlength{\unitlength}{#1}
\begin{picture}(35,30)
\allinethickness{1pt}
\put(5,5){\line(0,1){20}}
{\allinethickness{.8pt}
\put(2,15){\line(1,0){6}}
}
{\allinethickness{1pt}
\put(12,15){\vector(1,0){8}}
}
{\allinethickness{2pt}
\put(30,5){\line(0,1){20}}
}
{\allinethickness{1pt}
\qbezier(27,5)(27,10)(30,15)
\qbezier(30,15)(33,20)(33,25)
\qbezier(33,5)(33,10)(31,13)
\qbezier(29,17)(27,20)(27,25)

}
\end{picture}
}
\def\cpmovei#1{
\setlength{\unitlength}{#1}
\begin{picture}(60,20)
\put(0,0){\line(1,1){20}}
\put(20,0){\line(-1,1){20}}
\put(10,10){\circle{5}}
\put(17,17){\circle*{2}}

{\allinethickness{.8pt}
\put(25,10){\vector(1,0){10}}
\put(35,10){\vector(-1,0){10}}}

\put(40,0){\line(1,1){20}}
\put(60,0){\line(-1,1){20}}
\put(50,10){\circle{5}}
\put(43,3){\circle*{2}}
\end{picture}
}
\def\cpmoveii#1{
\setlength{\unitlength}{#1}
\begin{picture}(40,20)
\put(5,0){\line(0,1){20}}
{
\put(5.,7){\circle*{2}} 
\put(5.,13){\circle*{2}} }
{\allinethickness{.8pt}
\put(15,10){\vector(1,0){10}}
\put(25,10){\vector(-1,0){10}}}
\put(35,0){\line(0,1){20}}
\end{picture}
}
\def\cpmoveiii#1{
\setlength{\unitlength}{#1}
\begin{picture}(70,20)

\put(0,0){\line(1,1){20}}
\put(20,0){\line(-1,1){9}}
\put(9,11){\line(-1,1){9}}

\put(3,3){\circle*{2}}
\put(17,3){\circle*{2}}
\put(17,17){\circle*{2}}
\put(3,17){\circle*{2}}

{\allinethickness{.8pt}
\put(25,10){\vector(1,0){10}}
\put(35,10){\vector(-1,0){10}}}

\put(40,0){\line(1,1){20}}
\put(60,0){\line(-1,1){9}}
\put(49,11){\line(-1,1){9}}

\end{picture}
}

\begin{document}
\title[Converting virtual link diagrams to normal ones]
{Converting virtual link diagrams to normal ones}
\author{Naoko Kamada} 
\thanks{This work was supported by JSPS KAKENHI Grant Number 15K04879.}
\address{ Graduate School of Natural Sciences,  Nagoya City University\\ 
1 Yamanohata, Mizuho-cho, Mizuho-ku, Nagoya, Aichi 467-8501 Japan
}

\date{}
\begin{abstract} 
A virtual link diagram is called normal if the associated abstract link diagram is checkerboard colorable, and a virtual link is normal if it has a normal diagram as a representative.
In  this paper, we introduce a method of converting a virtual link diagram to a normal virtual link diagram by use of  the double covering technique. 
We show that the normal virtual link diagrams obtained from two equivalent virtual link diagrams are related by generalized Reidemeister moves and Kauffman flypes.
We obtain a numerical invariant of virtual knots by using our converting method.
\end{abstract}

\maketitle


\section{Introduction}

L. H. Kauffman  \cite{rkauD} introduced virtual knot theory, which is a generalization of knot theory based on  Gauss  diagrams and link diagrams in closed oriented surfaces. 
Virtual links correspond to  stable equivalence classes of  links in thickened surfaces \cite{rCKS,rkk}. 
A virtual link diagram is called normal if the associated abstract link diagram is checkerboard colorable ($\S$ \ref{secnormal}). A virtual link is called normal if it has a normal diagram as a representative. Every classical link diagram is normal, and hence 
the set of classical link diagrams is a subset of that of normal virtual link diagrams. The set of normal virtulal link diagrams is a subset of that of virtual link diagrams. 
The $f$-polynomial (Jones polynomial) is an invariant of a virtual link \cite{rkauD}. It is shown in \cite{rkn0} that the $f$-polynomial of a normal virtual link has a property that the $f$-polynomial of a classical link has. 
This property may make it easier to define Khovanov homology of virtual links as stated in O. Viro \cite{rviro}.

In  this paper, we introduce a method of converting a virtual link diagram to a normal virtual link diagram by use of  the double covering technique defined in \cite{rkk7}. 
We show that the normal virtual link diagrams obtained from two equivalent virtual link diagrams by our method 
are related by generalized Reidemeister moves and Kauffman flypes.
We obtain a numerical invariant of virtual knots by using the converting method.
\section{Definitions and main results}\label{secnormal} 
A {\it virtual link diagram\/} is a generically immersed loops with information of positive, negative or virtual crossing, on its double points.  A {\it virtual crossing\/} is an encircled double point without over-under information \cite{rkauD}. 
A {\it twisted link diagram\/} is a virtual link diagram, possibly with {\it bars\/} on arcs. 
%
%
A {\it virtual link\/} (or {\it twisted link}) is an equivalence class of virtual (or twisted ) link diagrams under Reidemeister moves and virtual Reidemeister moves (or Reidemeister moves, virtual Reidemeister moves and twisted Reidemeister moves) depicted in Figures~ \ref{fgmoves}. We call Reidemeister moves and virtual Reidemeister moves  {\it generalized Reidemeister moves}. 
\begin{figure}[h]
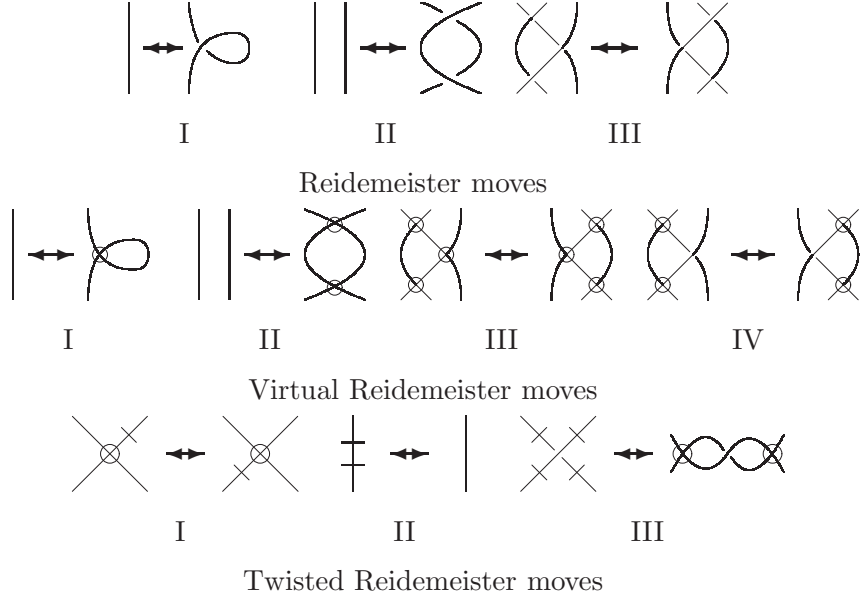

\centerline{
\begin{tabular}{ccc}
\rmoveio{.4mm}{1}&\rmoveiio{.4mm}{1}&\rmoveiiio{.4mm}{1}\\
I&II&III\\
\multicolumn{3}{c}{Reidemeister moves}
\end{tabular}}
\centerline{
\begin{tabular}{cccc}
\rmovevio{.4mm}{1}&\rmoveviio{.4mm}{1}&\rmoveviiio{.4mm}{1}&\rmovevivo{.4mm}{1}\\
I&II&III&IV\\
\multicolumn{4}{c}{Virtual Reidemeister moves}
\end{tabular}}
\centerline{
\begin{tabular}{ccc}
\rmovetti{.5mm}&\rmovetiio{.5mm}&\rmovetiiio{.5mm}{1}\\
I&II&III\\
\multicolumn{3}{c}{Twisted Reidemeister moves}
\end{tabular}}
\caption{Generalized Reidemeister moves and twisted Reidemeister moves}\label{fgmoves}
\end{figure}

We introduce the double covering diagram of a twisted link diagram. Let $D$ be a twisted link diagram. 
Assume that $D$ is on the right of the $y$-axis in the $xy$-plane and all bars are parallel to the $x$-axis with disjoint $y$-coordinates. Let  $D^*$ be  the twisted link diagram obtained from $D$ by reflection with respect to the $y$-axis and switching the over-under information of all classical crossings of $D$. Let $B=\{b_1,\dots, b_k\}$ be a set of bars of $D$ and we denote the bar of $D^*$ corresponding to $b_i$ by $b_i^*$.
See Figure~\ref{fgconvert1} (i). 
\begin{figure}[h]
\centering{
\includegraphics[width=5.5cm]{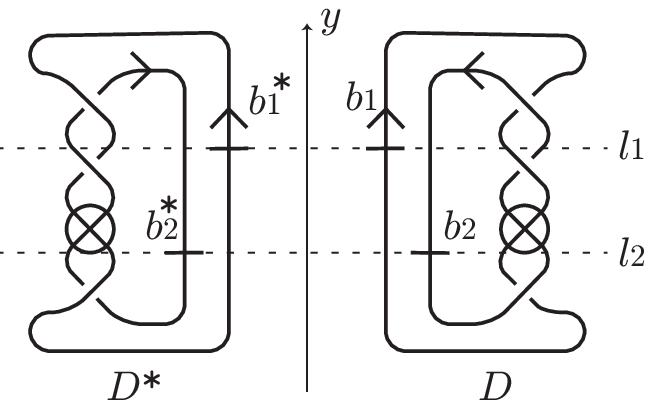}\hspace{1cm}
\includegraphics[width=5.cm]{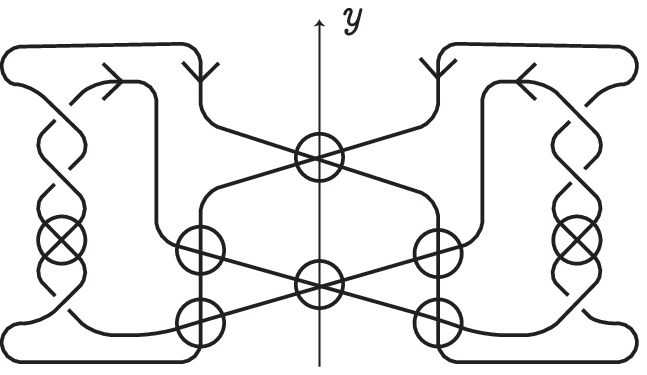}}\\
\centerline{(i)\hspace{6cm}(ii)}
\caption{The double covering of a twisted link diagram}\label{fgconvert1}
\end{figure}
For horizontal lines $l _1, \dots , l_k$ such that  $l_i$ contains $b_i$ and the corresponding bar $b_i^*$ of $D^*$, 
we  replace  each part of $D\amalg D^*$ in a neighborhood of $N(l_i)$ for each $i\in\{1,\dots, k\}$  as in Figure~\ref{fgconvert2}. We denote by $\phi(D)$ the virtual link diagram obtained this way.

For example, for the twisted link diagram $D$  depicted as in Figure~\ref{fgconvert1} (i), the virtual link diagram $\phi(D)$ is as in Figure~\ref{fgconvert1} (ii).
\begin{figure}[h]
\centering{
\includegraphics[width=8cm]{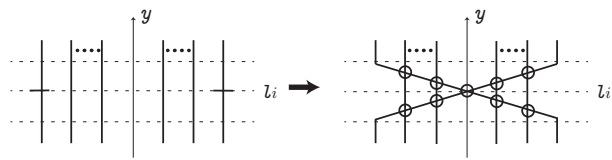}}
\caption{The replacement of diagram}\label{fgconvert2}
\end{figure}
We call this diagram $\phi(D)$ the {\it double covering diagram} of $D$.  
Then we have the followings.

\begin{thm}{\cite{rkk7}}\label{kkthm1}
Let $D_1$ and $D_2$ be twisted link diagrams.
If $D_1$ and $D_2$ are equivalent as a twiated link, then $\phi(D_1)$ and $\phi(D_2)$ are equivalent as a virtual link.
\end{thm}

An {\it abstract link diagram} ({\it ALD})  is a pair of  a compact surface $\Sigma$ and a link diagram $D$ on $\Sigma$ such that the underlying 4-valent graph 
$|D|$ is a deformation retract of $\Sigma$, denoted by $(\Sigma, D_{\Sigma})$. 

%
%
We obtain an ALD from a twisted link diagram $D$ as in Figure~\ref{fg:virabs}.
Such an  ALD is called the {\it ALD associated with} $D$.
Figure~\ref{fig:extwtdiag} shows twisted link diagrams and the ALDs associated with them. 
For details on abstract link diagrams and their relations to virtual links, refer to \cite{rkk}.
\begin{figure}[h]
\centerline{
\begin{tabular}{ccc}
\abscrs{.4mm}&\absvcrs{.4mm}&\abstwt{.4mm}
\end{tabular}
}
\caption{The correspondence from a twisted link diagrams to an ALD}\label{fg:virabs}
\end{figure}
\begin{figure}[h]
\centerline{
\begin{tabular}{cccc}
\includegraphics[width=2cm]{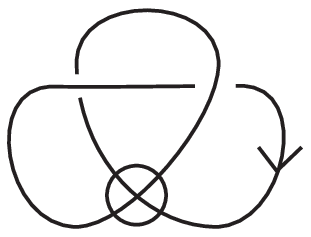}&
\includegraphics[width=2cm]{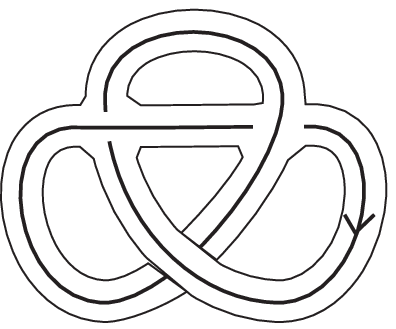}&
\includegraphics[width=2cm]{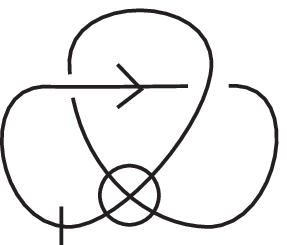}&
\includegraphics[width=2cm]{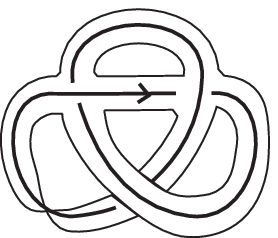}\\
(i)&(ii)&(iii)&(iv)\\
\end{tabular}}
\caption{Twisted link diagrams and  ALDa}\label{fig:extwtdiag}
\end{figure}
Let $D$ be a twisted link diagram and $(\Sigma, D_{\Sigma})$ the ALD associated with $D$. 
The diagram $D$ is said to be {\it normal} or {\it checkerboard colorable} if the regions of $\Sigma-|D_{\Sigma}|$ can be colored black and white such that colors of two adjacent regions are different. In Figure~\ref{fig:exccdiag}, we show an example of a normal diagram. A classical link diagram is normal. A twisted link is said to be {\it normal} if it has a normal twisted link diagram. Note that normality is not necessary to be 
preserved  under generalized Reidemeister moves. 
For example the virtual link diagram in the right of Figure~\ref{fgexdiag} is not normal and is equivalent to the trefoil knot diagram in the left which is normal. 
\begin{figure}[h]
\centerline{
\includegraphics[width=7cm]{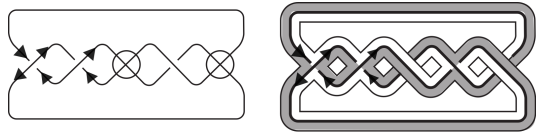}}
\caption{A normal twisted link  diagram and its associated ALD with a checkerboard coloring}\label{fig:exccdiag}
\end{figure}
\begin{figure}[h]
\centerline{
\includegraphics[width=5cm]{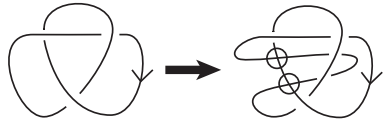}}
\caption{A diagram of a normal virtual link which is not normal}\label{fgexdiag}
\end{figure}
\begin{prop}\label{thm1}
For a normal twisted link diagram $D$, the double covering diagram of $D$  $\phi(D)$ is normal.
\end{prop}
\begin{proof}
Let $D$ be a normal twisted  link diagram $D$. The twisted link diagram $D^*$ is obtained from $D$ by reflection and switching all classical crossings of $D$ as the previous manner as in Figure~\ref{fgconvert4} (i). The  ALDs obtained from $D$ and $D^*$ can be colored as in Figure~\ref{fgconvert4} (ii). Then we see that $\phi(D)$ is normal as in the right of Figure~\ref{fgconvert4} (ii). 
\end{proof}
\begin{figure}[h]
\centering{
\includegraphics[width=11cm]{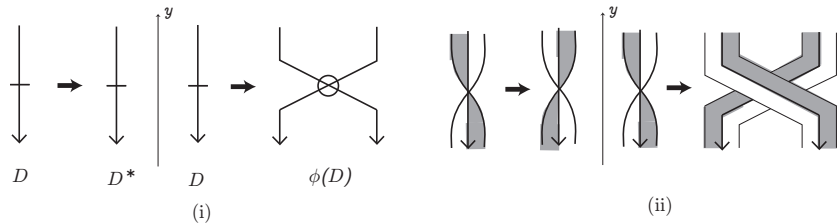} }
\caption{The converting diagram}\label{fgconvert4}
\end{figure}

%


%

H. Dye introduced the notion of cut points to a virtual link diagram in her talk presented in the Special Session 35, ``Low
Dimensional Topology and Its
Relationships with Physics", held in Porto, Portugal, June 10-13,
2015 as part of the 1st AMS/EMS/SPM Meeting.

Let $(D, P)$ be a pair of a virtual link diagram $D$ and a finite set $P$ of points on edges of $D$. 
We call the  ALD associated with the twisted link diagram obtained from $(D, P)$ by replacing all points of $P$ with bars,  the {\it ALD associated with} $(D, P)$.
 See Figure~\ref{fgcutpt} (ii) and (iii). 
If the ALD associated with $(D, P)$ is normal,  then we call the set of points $P$  a {\it cut system} of $D$ and call each point of $P$ a {\it cut point}.
For  the virtual link diagram  in Figure~\ref{fgcutpt} (i) we show an example of a cut system  in Figure~\ref{fgcutpt} (ii)  and the  ALD associated with it in Figure~\ref{fgcutpt} (iii). 

%
%

%
\begin{figure}[h]
\centerline{
\includegraphics[width=8cm]{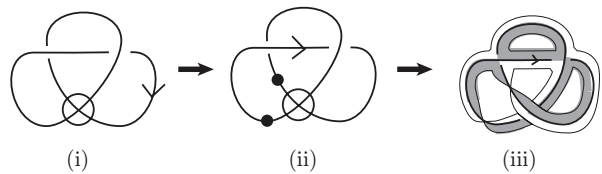}}
\caption{Example of cut points}\label{fgcutpt}
\end{figure}

A virtual link  diagram is said to admit an {\it alternate orientation} if it can be given an orientation  such that an  orientation of an edge switches at each classical crossing  as in Figure~\ref{fig:altori}. The virtual link diagram in Figure~\ref{fig:exccdiag} admits an alternate orientation. 
It is known that a virtual link diagram is normal (or checkerboard colorable) if and only if it admits an alternate orientation \cite{rkns}. 
\begin{figure}[h]
\centerline{
\includegraphics[width=5cm]{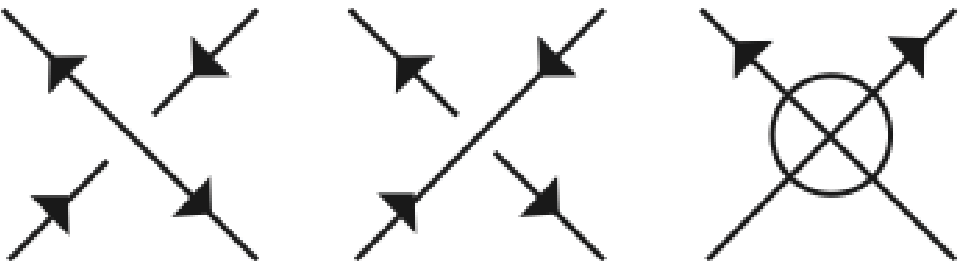}
}
\caption{Alternate orientation}\label{fig:altori}
\end{figure}
%


Note that a finite set $P$ of points on $D$ is a cut system if and only if $(D,P)$ admits an alternate orientation such that the orientations are as in Figure~\ref{fgalternateorip} at each crossing of $D$ and each point of $P$ (cf. \cite{rkn1,rkns}).
\begin{figure}[h]
\centerline{
\includegraphics[width=6cm]{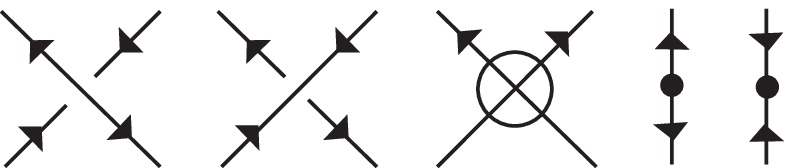}}
\caption{Alternate orientation}\label{fgalternateorip}
\end{figure}

The {\it canonical cut system} of a virtual link diagram $D$  is the set of points 
 that is obtained by 
giving two points in a neighborhood of each virtual crossing of $D$ as in Figure~\ref{fgVtoCC} (i). 
\begin{figure}[h]
\centerline{
\includegraphics[width=11cm]{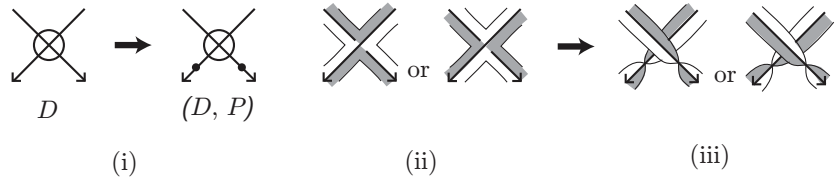}}
\caption{Canonical cut system of a virtual link diagram}\label{fgVtoCC}
\end{figure}
\begin{prop}\label{propcan}
The canonical cut system is a cut system.
\end{prop}

\begin{proof}
For a virtual link diagram $D$, let $D_C$ be a classical link diagram which is obtained from $D$ by replaceing all virtual crossings of $D$ with classical ones. Note that there is a checkerboard coloring for the ALD associated with $D_C$. At each classical crossing, the checkerboard coloring is as in Figure~\ref{fgVtoCC} (ii). 
Let $P$ be the canonical cut system of $D$. 
The ALD associated with $(D, P)$  is checkerboard colorable such that it's coloring is inherited from that of $D_C$ as in Figure~\ref{fgVtoCC} (ii) and (iii).
\end{proof}

Dye introduced the cut point moves  depicted in Figure~\ref{fgcutptmove} and asked
whether two cut systems of a virtual link diagram $D$ are related by a sequence of cut point moves. The following theorem answers it.

\begin{figure}[h]
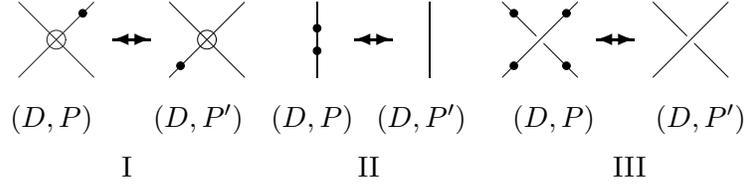

\centerline{
\begin{tabular}{ccc}
\cpmovei{.5mm}&\cpmoveii{.5mm}&\cpmoveiii{.5mm}\\
$(D,P)$\hspace{.8cm}$(D, P')$&$(D,P)$\hspace{.3cm}$(D, P')$&$(D,P)$\hspace{.8cm}$(D, P')$\\
I&II&III\\
\end{tabular}
}
\caption{The cut point moves}\label{fgcutptmove}
\end{figure}
\begin{thm}\label{thmc1}
For a virtual link diagram $D$, two cut systems of  $D$ are related by a sequence of cut point moves I, II and III.
\end{thm}
\begin{proof}
Let $P$ and $P'$ be two cut systems of $D$. 
For $(D,P)$ and $(D,P')$, give  alternate orientations ${\cal O}$ and ${\cal O}'$, respectively. 
Let $c_1, c_2, \dots , c_m$ be classical crossings of $D$ where the orientations of edges of $\cal O$ are different from those of 
$\cal O'$. 
Apply cut point moves III  at $c_1, c_2, \dots , c_m$ to $(D,P)$, then we obtain the cut system $P''$ of $D$. 
There is an alternate orientation of $(D,P'')$, say $\cal O''$,  such that each classical crossing in $D$ admits the same 
orientation to that of $\cal O'$. 
Applying some cut point moves I to $P''$, we have a cut system $P'''$ such that for each edge $e$ of $D$, the number of cut points on $e$ in $(D,P''')$ is congruent to that in $(D,P')$ modulo $2$. Then $(D,P')$ is obtained from $(D,P''')$ by cut point moves I and II. 
\end{proof}

\begin{cor}[H. Dye]\label{cor1}
For any virtual link diagram with a cut system,  the number of cut points is 
even.
\end{cor}
\begin{proof}
The number of cut points of  the canonical cut system is even. Since cut point moves do not change the parity of the number of cut points, we obtain the result.
\end{proof}


Let $(D,P)$ be a virtual link diagram with a cut system. 
We replace all points of $P$ with bars. Then we obtain a normal twisted link diagram. We denote such a map from the set of virtual link diagrams with cut systems to that of twisted link diagrams by $t$. 
We denote the image of $(D, P)$ under $t$ by $t(D, P)$. 
The double covering of  $t(D,P)$ is normal from Proposition~\ref{thm1} since $t(D, P)$ is normal. 
For a virtual link diagram with a cut system $(D,P)$  the double covering diagram of $t(D, P)$ is called the {\it converted normal diagram} of $(D, P)$, denoted by $\phi(D, P)$. 
The local replacement of a virtual link diagram depicted in Figure~\ref{fgkflype} is called a {\it Kauffman flype} or a {\it K-flype}.
 If a virtual link diagram $D'$ is obtained from $D$ by a finite sequence of generalized Reidemeister moves and K-flypes, then they are said to be {\it K-equivalent}.
\begin{figure}[h]
\centering{
\includegraphics[width=3cm]{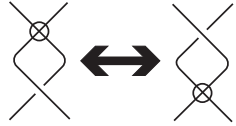} }
\caption{Kauffman flype}\label{fgkflype}
\end{figure}
\begin{rem}
The $f$-polynomials of K-equivalent virtual link diagrams are the same \cite{rkauD}. 
For a virtual link diagram of $D$, if a virtual link diagram $D'$ is obtained from $D$ by a K-flype at a classical crossing $c$, then the sign of the corresponding classical crossing $c'$ of $D'$ is the same as that of $c$. 
If $D$ is normal, then $D'$ is normal. 
\end{rem}

The following is our main theorem.
\begin{thm}\label{thm2}
Let $(D,P)$ and $(D',P')$ be virtual link diagrams with cut systems. If $D'$ is equivalent (or K-equivalent) to $D$, 
then the converted normal diagram $\phi(D,P)$ is K-equivalent to $\phi(D',P')$.
\end{thm}
For a 2-component of virtual link diagram  $D$, the half of the sum of signs of non self-crossings of $D$ is said to be the {\it linking number} of $D$.  The following is clear.
\begin{prop}\label{prop4}
The linking number is invariant under the generalized Reidemeister moves and K-flypes.
\end{prop}
We have the following theorems from our main theorem and Proposition~\ref{prop4}.
\begin{thm}\label{thmk1}
Let $(D,P)$ be a virtual knot diagram with a cut system. 
Then $\phi(D,P)$ is a 2-component virtual link diagram and the linking number of $\phi(D,P)$ is an invariant of the virtual knot represented by $D$.
\end{thm}

The odd writhe is a numerical invariant of virtual knots \cite{rkauE}. We recall the definition of the odd writhe in Section~\ref{property}. 
\begin{thm}\label{thmk2}
Let $(D,P)$ be a virtual knot diagram with a cut system. The linking number of $\phi(D,P)$ coincides to the odd writhe of $D$.
\end{thm}

\section{Proof of Theorem \ref{thm2}}
Theorem~\ref{thm2} is obtained from Lemmas~\ref{lem2} and \ref{lem1} stated below. 
\begin{lem}\label{lem2}
Let $D$ be a virtual  link diagram. Suppose that $P$ and  $P'$ are two cut systems of $D$.
Then the converted normal diagrams $\phi(D,P)$ and $\phi(D,P')$ are K-equivalent.
\end{lem}
\begin{proof}
Let $D$ be a virtual  link diagram with a cut system $P$. Suppose that  $P'$ is a cut system of $D$ obtained from $P$ by one of cut point moves I or II in Figure~\ref{fgcutptmove}. 
Then $t(D, P')$ is obtained from $t(D, P)$ by a twisted Reidemeister move I or II, respectively. Thus by Theorem~\ref{kkthm1}, $\phi(D, P)$ and $\phi(D, P')$ are equivalent. 
If $P'$ is related to $P$ by a cut point move III in Figure~\ref{fgcutptmove}, then $\phi(D, P)$ and $\phi(D,P')$ are related by virtual Reidemeister moves and K-flypes as in Figure~\ref{fgprflem21}. 
If the orientations of some strings of a virtual link diagram are different from those in Figure~\ref{fgprflem21}, we have the result by a similar argument. 
\end{proof}
\begin{figure}[h]
\centerline{\parbox{5cm}{
 \includegraphics[width=5cm]{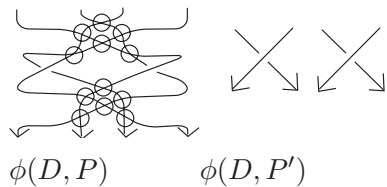}\\
\hspace{.6cm}$\phi(D,P)$\hspace{1.2cm}$\phi(D,P')$
}}
\caption{Converted normal diagrams related by a cut point move III}\label{fgprflem21}
\end{figure}
\begin{lem}\label{lem1}
Let $(D,P)$ and $(D',P')$ be virtual link diagrams with canonical cut systems.
If $D'$ is equivalent (or K-equivalent) to $D$, then $\phi(D,P)$ and $\phi(D',P')$ are K-equivalent.
\end{lem}
\begin{proof}
Let $D_1$ be a virtual  link diagram with the canonical cut system $P_1$. Suppose that a virtual link diagram $D_2$ is obtained from $D_1$ by one of generalized Reidemeister moves or K-flype  
and $P_2$ is the canonical cut system of $D_2$.
If $D_2$ is related to $D_1$ by one of Reidemeister moves, then $\phi(D_1, P_1)$ and $\phi(D_2,P_2)$ are related by two Reidemeister moves. 
As in Figure~\ref{fgprflem11} (i) or (ii), suppose that  $D_2$ is related to $D_1$  by a virtual Reidemeister move I or II and let $P_2'$ be the cut system obtained from $P_2$ by cut point moves I and II as in the figure. 
By Lemme~\ref{lem2},  $\phi(D_2,P_2)$  are $\phi(D_2,P_2')$ are equivalent.
By Theorem~\ref{kkthm1}, $\phi(D_1,P_1)$  are $\phi(D_2,P_2')$ are equivalent since $t(D_1, P_1)$ and $t(D_2,P_2')$ are related by virtual Reidemeidter move I or II, which means that $\phi(D_1,P_1)$  and  $\phi(D_2,P_2)$ are equivalent. 
As in Figure~\ref{fgprflem11} (iii), suppose that  $D_2$ is related to $D_1$  by a virtual Reidemeister move III and let $P_1'$ (or $P_2'$) be the cut system obtained from $P_1$ (or $P_2$) by cut point moves I and II as in the figure. 
By Lemme~\ref{lem2},  $\phi(D_1,P_1)$  (or $\phi(D_2,P_2)$ ) and $\phi(D_1,P_1')$ (or $\phi(D_2,P_2')$) are equivalent.
By Theorem~\ref{kkthm1} $\phi(D_1,P_1')$  are $\phi(D_2,P_2')$ are equivalent since $t(D_1,P_1')$ and $t(D_2,P_2')$ are related by virtual Reidemeidter move III, 
which means that $\phi(D_1,P_1)$  and $\phi(D_2,P_2)$ are equivalent. 
As in Figure~\ref{fgprflem11}  (iv), suppose that  $D_2$ is related to $D_1$  by a virtual Reidemeister move IV and let $P_1'$ (or $P_2'$) be the cut system obtained from $P_1$ (or $P_2$) by cut point moves I and II  
(or cut point moves I , II and III) as in the figure. 
By the similar reason,  $\phi(D_1,P_1)$  and $\phi(D_2,P_2)$ are K-equivalent. 
\begin{figure}[h]
\centering{
\begin{tabular}{ccc}
\includegraphics[width=3cm]{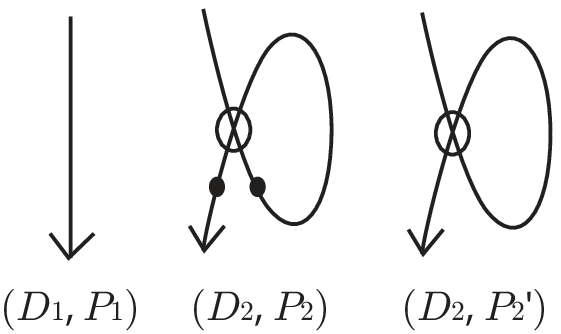} &\includegraphics[width=3.5cm]{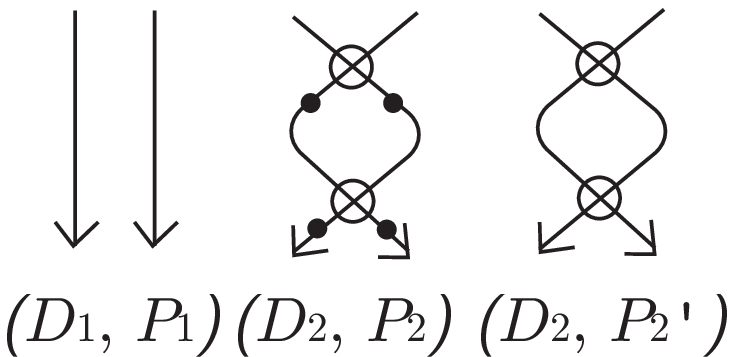}&
\includegraphics[width=4cm]{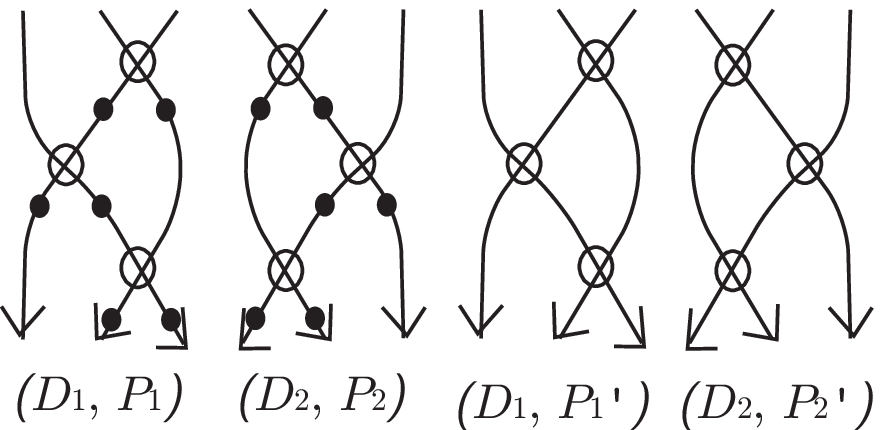}\\
{\small Virtual Reidemeister move I }&{\small Virtual Reidemeister move II}&{\small Virtual Reidemeister move III}\\
(i)&(ii)&(iii)\\
\includegraphics[width=4cm]{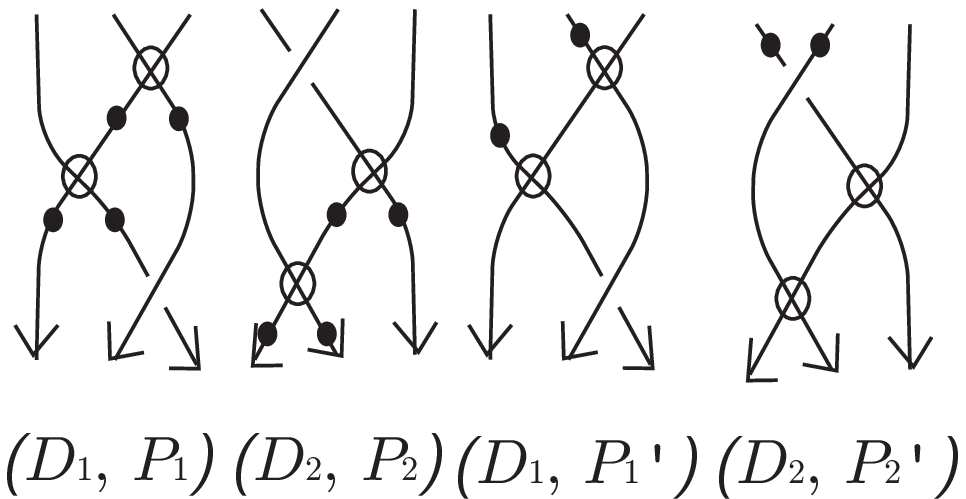}& \multicolumn{2}{c}{\includegraphics[width=7cm]{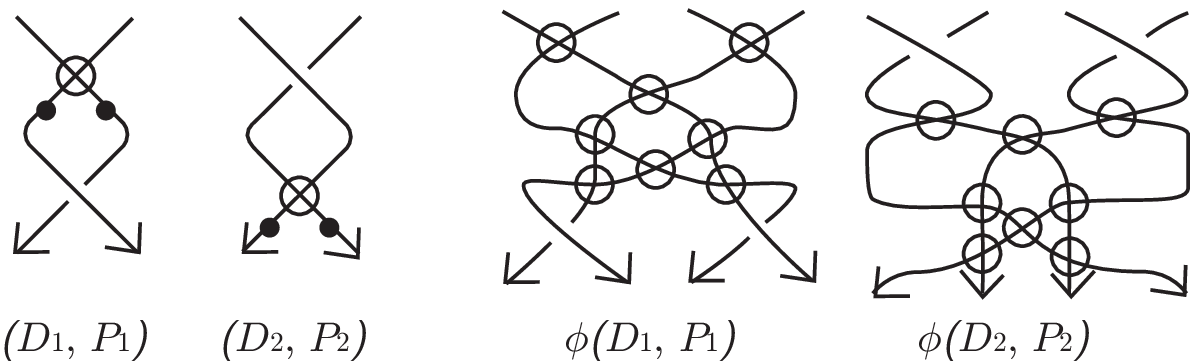}}\\
{\small Virtual Reidemeister move IV}&\multicolumn{2}{c}{K-flype}\\
(iv)&\multicolumn{2}{c}{(v)}\\
\end{tabular}}
\caption{Diagrams related by a virtual reidemeister moves and a K-flype}\label{fgprflem11}
\end{figure}
If $D_1$ is related to $D_2$ by a K-flype, then $\phi(D_1, P_1)$ and $\phi(D_2,P_2)$ are related by some virtual Reidemeister moves in  Figure~\ref{fgprflem11} (v). 
In Figure~\ref{fgprflem11}, if the orientation of some strings of virtual link diagram $D_i$ are different from those of it,  we have the result by a similar argument. 
\end{proof}
%

\section{Proof of Theorems~\ref{thmk1} and \ref{thmk2}}\label{property}

Let $D$ be a virtual link diagram. The {\it Gauss diagram} of $D$ is a set of oriented circles such that each component is the preimage of $D$ with oriented chords each of which corresponds to a classical crossing and  its starting point (or ending point) indicates an over path  (or an under path) of the  classical crossing. Each chord is equipped with a sign of the corresponding classical crossing. 
The Gauss diagram in Figure \ref{fgexgauss1} (i)  is that of the virtual knot diagram in Figure \ref{fig:extwtdiag} (i). 
\begin{figure}[h]
\centering{
\includegraphics[width=10cm]{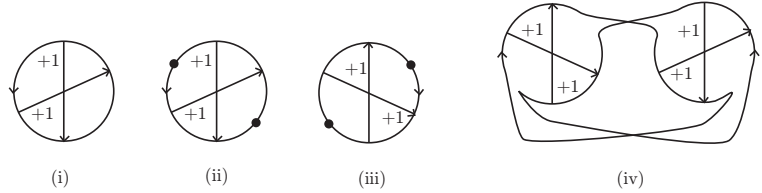} }
\caption{The Gauss diagrams}\label{fgexgauss1}
\end{figure}
For  a virtual link diagram $D$ with a set of points, $P$, the Gauss diagram with points of $(D,P)$ is obtained from the Gauss diagram of $D$ by adding points on arcs which correspond to the points of $D$. In Figure~\ref{fgexgauss1} (ii) we see  the Gauss diagram with a set of points of the virtual link diagram with a set of points in Figure~\ref{fgcutpt} (ii). We denote the Gauss diagram of $D$ with a set of points, $P$ by $G(D, P)$. 
Let $(D^*,P^*)$ be the virtual link diagram with a cut system which is obtained from $(D,P)$ by reflection with respect to $y$-axis and switching the over-under information of all classical crossings of $D$.
The Gauss diagram with points of $(D^*,P^*)$ is obtained from the Gauss diagram of $(D,P)$ by reflection with respect to $y$-axis and revering all orientations of chords.  
For example, the Gauss diagram with points in Figure~\ref{fgexgauss1} (iii) is the Gauss diagram of the virtual knot diagram with a set of points obtained from a virtual knot diagram  in Figure \ref{fig:extwtdiag} (ii) by the reflection with respect to $y$-axis and switching all classical crossings. 
The Gauss diagram of $\phi(D,P)$ is obtained from the Gauss diagrams of $D\amalg D^*$ by a local replacement around each point $p$ of the Gauss diagram of $D$ and around the corresponding point $p^*$ of that of $D^*$ as in  Figure~\ref{fggaussrep}.  For a virtual knot diagram of $D$ with a set of points on edges $P$, we denote the Gauss diagram of $\phi(D,P)$ by $G(\phi(D, P))$. 
The Gauss diagram in Figure~\ref{fgexgauss1} (iv) is $G(\phi(D,P))$ for the $(D, P)$ depicted in Figure~\ref{fgcutpt} (ii) by the map $\phi$. 
%
%
\begin{figure}[h]
\centering{
\includegraphics[width=6cm]{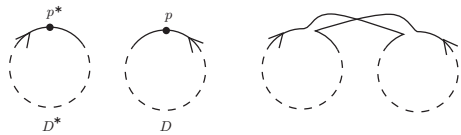} }
\caption{Replacement of  Gauss diagams}\label{fggaussrep}
\end{figure}
For the Gauss diagram of $(D\amalg D^*$, $P\amalg P^*)$, suppose that $p_1, \dots, p_n$ are points in $P$ such that the point $p_{i+1}$ follows the point $p_{i}$ along the orientaion of $D$ and the point $p_i^*$ in $P^*$  is symmetric to the point $p_i$. Let $A_i$ (or $A_i^*$) be the arc of the Gauss diagram of $G(D\amalg  D^*, P\amalg P^*)$ between two points $p_i$ and $p_{i+1}$ (or $p_i^*$ and $p_{i+1}^*$), and  the arc between two points $p_n$ and $p_{1}$ (or $p_n^*$ and $p_{1}^*$) is $A_n$ (or $A_n^*$). Note that $A_i$ is symmetric to $A_i^*$. We also denote an arc of $G(\phi(D,P))$ which corresponds to  $A_i$ or $A_i^*$  of $G(D\amalg D^*, P\amalg P^*)$ by $\widetilde{A_i}$ or $\widetilde{A_i^*}$, respectively. Here $\widetilde{A_i}$ (or $\widetilde{A_i^*}$) is the arc in $G(\phi(D,P))$ which is obtained from $A_i$ (or $A_i^*$) by removing a regular neighborhood of $p_i$ and $p_{i+1}$(or  $p_i^*$ and $p_{i+1}^*$).

We proof the following lemmas in stead of Theorem~\ref{thmk1}.
\begin{lem}\label{lem3}
Let $D$ be a virtual knot diagram with  a set of points on edges $P=\{p_1,\dots, p_{2n}\}$ for a positive integer $n$. Then  $\phi(D,P)$ is a 2-component virtual link diagram $D_1\cup D_2$. Furthermore if an arc $A_i$ belongs to $G(\phi(D, P))|_{D_1}$ (or $G(\phi(D, P))|_{D_2}$), then $A_{i+1}$ and $A_{i}^*$ belong  to $G(\phi(D, P))|_{D_2}$ (or $G(\phi(D, P))|_{D_1}$).
\end{lem}
\begin{proof}
We use the induction on $n$. 
Suppose that $n=1$, i.e., $D$ is a virtual knot diagram with 2 points $p_1$ and $p_2$. The Gauss diagram $G(\phi(D, P))$ is depicted as in Figure~\ref{fglemg1}, where the bold line and the thin line indicate the different components and we dropped all chords in the figure. In this case $\phi(D, P)$ is a 2-component virtual link diagram. Two arcs $\widetilde{A_1}$  and $\widetilde{A_2^*}$ are in one component of $G(\phi(D, P))$, and $\widetilde{A_{2}}$ and $\widetilde{A_{1}^*}$ are in the other.
\begin{figure}[h]
\centering{
\includegraphics[width=6cm]{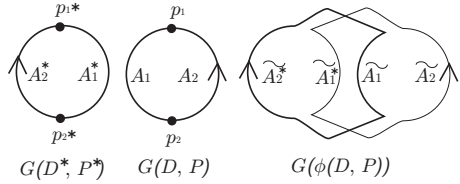} }
\caption{$G(\phi(D,P))$ of $D$ a pair of point}\label{fglemg1}
\end{figure}
Suppose that the statement is hold if the number of points is less than $2n$.
We assume that $D$ is a virtual knot diagram with $2n$ points $p_1,\dots, p_{2n}$. 
We apply the replacement as in Figure~\ref{fggaussrep} to $2n-2$ points $p_1,\dots, p_{2n-2}$ and 
$p_1^*,\dots, p_{2n-2^*}$. Then we obtain the Gauss diagram $G$ with $4$ points $p_{2n-1}, p_{2n}$ and  $p_{2n-1}^*, p_{2n}^*$. 
By the hypothesis, the Gauss diagram $G$  is  depicted as in Figure~\ref{fglemg2} (i), where two arcs ${A_{2n-1}}$ and ${A_{2n}}$ (or two points $p_{2n-1}$ and $p_{2n}$) are in one component of $G$ and two arcs ${A_{2n-1}^*}$ and ${A_{2n}^*}$ (or two points $p_{2n-1}^*$ and $p_{2n}^*$) are in  the other. If an arc $\widetilde{A_i}$ is in  one component of $G$,  $\widetilde{A_i^*}$ (or $\widetilde{A_{i+1}}$) is in the other for  $i\ne 2n-1$ by the hypothesis.  
By applying the replacement in Figure~\ref{fggaussrep} to two pairs of points $p_{2n-1}$ and $p_{2n-1}^*$ and $p_{2n}$ and $p_{2n}^*$ of  the Gauss diagram $G$, we have the Gauss diagram as in Figure~\ref{fglemg2} (ii). Therefore we have the result.
\begin{figure}[h]
\centering{
\includegraphics[width=12cm]{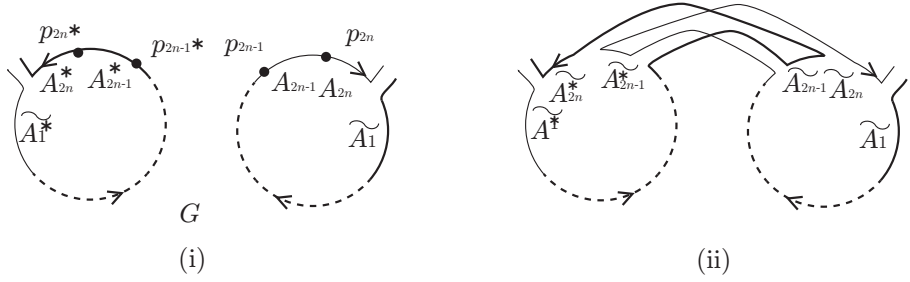} }
\caption{Replacement of  Gauss diagams}\label{fglemg2}
\end{figure}
\end{proof}
Thus we have the following lemma
\begin{lem}\label{thm5}
Let $(D_1, P_1)$ and $(D_2, P_2)$ be  virtual knot diagrams with cut systems. If  $D_1$ and $D_2$ are equivalent (or K-equivalent), then the linking number of $\phi(D_2,P_2)$ is equal to that of $\phi(D_1,P_1)$.
\end{lem}
\begin{proof}
By Theorem~\ref{thm2}, $\phi(D_1, P_1)$ and $\phi(D_2, P_2)$ are K-equivalent. Hence by Proposition~\ref{prop4}, they have the same linking numbers.
\end{proof}
For a virtual knot $K$ and its diagram $D$ with a cut system $P$, we denote  the linking number of $\phi(D,P)$ by $\mathrm{lk}_N(K)$ or $\mathrm{lk}_N(D)$. It dose not depend on the choice of $P$ by Lemma~\ref{thm5}.

Let $D$ be a virtual knot diagram and $G$ be a Gauss diagram of $D$. 
For a classical crossing $c$, we denote by $\gamma_c$ the chord of $G$ corresponding to $c$. The endpoints of $\gamma_c$ divides the circle of $G$ into 2 arcs. We denote  the arcs by $I_c$ and $I'_c$ where $I_c$ is the arc which starts from the tail of $\gamma_c$ and terminates at the head. 
A classical crossing  $c$ of $D$ is said to be   {\it odd} if there are an odd number of endpoints of chords of $G$ on $I_c$. The {\it odd writhe} of $D$ is the sum of signs of odd crossings of $D$．Note that if a virtual knot diagram is normal, all classical crossings are not odd.

\begin{thm}[\cite{rkauE}]\label{odd1}
The odd writhe is an invariant of virtual knots.
\end{thm}

\begin{proof}[Proof of Theorem~\ref{thmk2}]
Let $D$ be a virtual knot diagram and $P$ be a cut system of $D$.  
It is sufficient to show that 
odd crossings of $D$  correspond  to non self classical crossings  of $\phi(D,P)$. 
Since $(D,P)$ admits an alternate orientation, the circle of the Gauss diagram $G(D,P)$  of $(D,P)$ admits an alternate orientation such that  one endpoint of each chord is 
a sink of the orientations and the other is a source. This implies the following condition: 
\begin{itemize}
\item[($\ast$)] 
For any classical crossing $c$ of $(D, P)$,
 the sum of the number of cut points and that of endpoints of chords appearing on the arc $I_c$, is even. 
\end{itemize} 
Let $A_1, A_2, \dots $ be the arcs obtained by cutting the circle of $G(D,P)$ along the cut points.  We assume that 
 the arc $A_{i+1}$ appears after the arc $A_{i}$ along the orientation of $D$. 
We also denote  an arc of $G(\phi(D,P))|_D$ which corresponds to $A_i$ by $\widetilde{A_i}$. 
For a classical  crossing $c$ of $(D, P)$, the classical crossing corresponding to $c$ in $\phi(D, P)|_D$ is denoted by $\tilde{c}$. 
Let $c$ be  an odd crossing of $(D, P)$. 
Suppose that  one endpoint of  $\gamma_c$  
is on $A_k$ and the other endpoint  is on $A_j$. By definition, there are an odd number of endpoints of chords on $I_c$ in  $G(D,P)$. 
By the condition ($\ast$) above, we see that there are an odd number of cut points on $I_c$ in  $G(D,P)$.  
Then we see that $k$ is not congruent to  $j$ modulo $2$.  By Lemma~\ref{lem3} $\widetilde{A_k}$ and $\widetilde{A_j}$ in  $G(\phi(D,P))|_D$ are in the different components of $G(\phi(D,P))$. Thus the classical crossing $\tilde{c}$ is a non self classical crossing of $\phi(D,P)$.
\end{proof}

We show some property of $\mathrm{lk}_N(D)$. They are also obtained from the property of the odd writhe.

\begin{cor}[\cite{rkauE}]\label{prop3}
Let $K$ be a normal virtual knot. Then $\mathrm{lk}_N(K)$ is zero.
\end{cor}
\begin{proof}
Let $D_N$ be a normal knot diagram of $K$. Any cut system of $D_N$ is related to an empty set by some cut point moves since $D_N$ is normal. The virtual link diagram $\phi(D_N,\emptyset)$ is the disjoint union of $D_N$ and $D^*_N$, $D_N \amalg D_N^{*}$. Thus we see that $\mathrm{lk}_N(K)$ is zero  by Proposition~\ref{prop4}, since the linking number of $D_N \amalg D_N^{*}$ is zero.
\end{proof}
Let $D$ be a virtual knot diagram. The virtual knot diagram obtained from $D$ by switching the over-under information of all classical crossing (or by reflection) is denoted by $D^{\sharp}$ (or $D^{\dag}$).
\begin{cor}[\cite{rkauE}]\label{prop8}
Let $D$ be a virtual knot diagram. If $\mathrm{lk}_N(D)$ is not zero, then $D$ is not equivalent to $D^{\sharp}$ (or $D^{\dag})$. 
\end{cor}
\begin{proof}
It is clear that  $\mathrm{lk}_N(D^{\sharp})=\mathrm{lk}_N(D^{\dag})=-\mathrm{lk}_N(D)$.
\end{proof}

For example, the virtual knot presented by the diagram $D$ in Figure~\ref{fgexapli1}, is not normal by Corollary~\ref{prop3}, since $\mathrm{lk}_N(D)=-2$. By Corollary~\ref{prop8},  $D$ is not equivalent to $D^{\sharp}$ (or $D^{\dag})$.
\begin{figure}[h]
\centering{
\includegraphics[width=10cm]{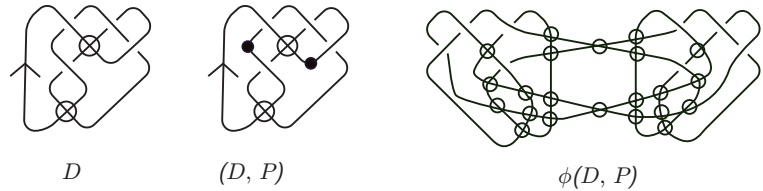} }
\caption{A non normal virtual link}\label{fgexapli1}
\end{figure}

\vspace{8pt}
{\bf Acknowledgement}

The author  would like to thank Seiichi Kamada for his useful suggestion.


\begin{thebibliography}{88}

\bibitem{rBor}
M. O. Bourgoin,
{\it Twisted link theory\/},
Algebr. Geom. Topol. 8 (2008) 1249--1279. 


\bibitem{rCKS}
J. S. Carter, S. Kamada and M. Saito,
{\it Stable equivalence of knots on surfaces and virtual knot cobordisms\/},
J. Knot Theory Ramifications  11 (2002) 
 311--322. 



\bibitem{rkn0} 
N. Kamada, 
{\it On the Jones polynomials of checkerboard colorable virtual knots}, 
Osaka J Math., 39 (2002), 325--333. 

\bibitem{rkn1} 
N.~Kamada,  
{\it On twisted knots\/},     
Contemp. Math., to appear.



\bibitem{rkk} 
N.~Kamada and S.~Kamada,  
{\it Abstract link diagrams and virtual knots\/},     
J. Knot Theory Ramifications {9} (2000) 
93-106.


\bibitem{rkk7} 
N.~Kamada and S.~Kamada,  
{\it Double coverings of twisted links\/},  
preprint (arXiv:1510.03001).

\bibitem{rkns}
N. Kamada, S. Nakabo and S. Satoh, 
{\it  A virtualized skein relation for Jones polynomials}, 
Illinois~J.~Math. {46} (2002),  467--475. 
 
\bibitem{rkauD}
L.~H.~Kauffman, 
{\it Virtual knot theory\/},
European~J.~Combin. 
{20} (1999) 663--690. 

\bibitem{rkauE}
L.~H.~Kauffman, 
{\it A self-linking invariant of virtual knots\/}. Fund. Math., 184 (2004),  135--158.









\bibitem{rviro} O. Viro, 
{\it Virtual Links, Orientations of Chord Diagrams and Khovanov Homology}, 
Proceedings of 12th G\"{o}kova Geometry-Topology Conference, (2007) 184--209.

\end{thebibliography}
 \end{document}